\documentclass[preprint]{imsart}

\input{header.tex}

\begin{document}

\UseSection  
\setcounter{tocdepth}{2}    

\begin{frontmatter}
  \title{Limit laws for self-loops and multiple edges\\in the configuration
    model} \runtitle{Self-loops and multiple edges in the CM}

\begin{aug}
\author{\fnms{Omer} \snm{Angel}
\ead[label=e1]{angel@math.ubc.ca}}

\address{Department of Mathematics,\\
University of British Columbia,\\
Vancouver, BC, V6T 1Z2, Canada
\printead{e1}}

\author{\fnms{Remco} \snm{van der Hofstad}
\ead[label=e2]{r.w.v.d.hofstad@tue.nl}}

\address{Department of Mathematics and
    Computer Science,\\ Eindhoven University of Technology, P.O.\ Box 513,\\
    5600 MB Eindhoven, The Netherlands.
\printead{e2}}

\author{\fnms{Cecilia} \snm{Holmgren} 
\ead[label=e3]{cecilia.holmgren@math.uu.se}}

\address{ Department of Mathematics,\\ Uppsala University, \\ SE-75310, Uppsala, Sweden.
\printead{e3}}

\runauthor{Angel, Van der Hofstad and Holmgren}

\affiliation{University of British Columbia, University of Eindhoven and
  Stockholm University}
\end{aug}

\begin{abstract}
  We consider self-loops and multiple edges in the configuration model as
  the size of the graph tends to infinity.  The interest in these random
  variables is due to the fact that the configuration model, conditioned on
  being simple, is a uniform random graph with prescribed degrees.
  Simplicity corresponds to the absence of self-loops and multiple edges.

  We show that the number of self-loops and multiple edges converges in
  distribution to two independent Poisson random variables when the second
  moment of the empirical degree distribution converges.  We also provide
  an estimate on the total variation distance between the number of
  self-loops and multiple edges and the Poisson limit of their sum.  This
  revisits previous works of Bollob\'as, of Janson, of Wormald and others.
  The error estimates also imply sharp asymptotics for the number of simple
  graphs with prescribed degrees.

  The error estimates follow from an application of Stein's method for
  Poisson convergence, which is a novel method for this problem.  The
  asymptotic independence of self-loops and multiple edges follows from a
  Poisson version of the Cram\'er-Wold device using thinning, which is of
  independent interest.

  When the degree distribution has infinite second moment, our general
  results break down.  We can, however, prove a central limit theorem for
  the number of self-loops, and for the multiple edges between vertices of
  degrees much smaller than the square root of the size of the graph, or
  when we truncate the degrees similarly.  Our results and proofs easily
  extend to directed and bipartite configuration models.
\end{abstract}

\begin{keyword}[class=AMS]
\kwd[Primary ]{60K35, 60K37, 82B43}
\end{keyword}

\begin{keyword}
\kwd{Configuration model, self-loops, multiple edges, Chen-Stein Poisson approximation}
\end{keyword}

\tableofcontents
\end{frontmatter}

\section{Introduction and motivation}

\subsection{Models and results}
\label{sec-mod}

We consider the configuration model $\CMnd$, with degrees
$\bfdit=(d_i)_{i\in[n]}$.  The configuration model (CM) is a random graph
with vertex set $[n]:=\{1,2,\ldots,n\}$ and with prescribed degrees.  Let
$\bfdit = (d_1,d_2,\ldots, d_n)$ be a given {\it degree sequence}, i.e., a
sequence of $n$ positive integers.  The total degree, denoted $\ell_n$ is 
\eqn{
  \ell_n = \sum_{i\in [n]} d_i,
}
and is assumed to be even.  The CM on $n$ vertices with degree sequence
$\bfdit$ is constructed as follows: start with $n$ vertices and $d_i$
half-edges adjacent to each vertex $i \in [n]$.  The $\ell_n$ half-edges
are matched in pairs in a uniformly random manner to form the edges of the
graph.

Algorithmically, the CM may be sampled as follows.  Randomly choose a pair
of half-edges, match the chosen pair together to form an edge and remove
the two half-edges.  Continue until all half-edges are paired.  We denote
the resulting graph on $[n]$ by $\CMnd$, with corresponding edge set
$\EE_n$.  Although self-loops may occur due to the pairing of half-edges
that are incident to the same vertex, in many cases these become rare
compared to the total degree as $n\to\infty$ (see e.g.\ \cite{Boll01,
  Jans06b}).  The same applies to multiple edges.  We say that $\CMnd$ is
{\em simple} when it has no self-loops nor multiple edges.

\medskip

In this paper, we investigate limit laws for the number of self-loops and
multiple edges.  Specifically, we study the random vector $(S_n,M_n)$,
which is defined as
        \eqan{
	\label{Sn-Mn-def}
	S_n &= \sum_{i\in [n]} X_{ii},
	&
	M_n &= \sum_{1\leq i<j\leq n} \binom{X_{ij}}{2}.
        }
Here, for $i,j\in[n]$, $X_{ij}$ denotes the number of edges between
vertices $i$ and $j$.  For clarity, note that we have
        \eqn{
	d_i = X_{ii} + \sum_{j\in[n]} X_{ij} = 2X_{ii} + \sum_{j\neq i} X_{ij}.
        }
We note that $M_n$ is not precisely equal to the number of multiple edges. This number instead may be written as
\eqn{ 
        \label{number-multiple-edges}
        \widetilde M_n =\sum_{1\leq i<j\leq n} (X_{ij}-1)_+,
}
where $x_+ = \max\{0,x\}$.  However, $M_n=0$ precisely when there are no
self-loops, i.e., when $\widetilde M_n=0$.  
Moreover, if $X_{ij}=2$ then the pair $i,j$ contributes $1$ to $M_n$, so $M_n=\widetilde M_n$ in
the absence of triple edges between vertices.

\medskip

Furthermore, we let $\lambdanS$ and $\lambdanM$ be the means of the random variables $S_n$ and $M_n$, i.e.,
\eqan{
	\label{lambda-defs}
	\lambdanS &= \expec[S_n],  & \lambdanM &= \expec[M_n]. \\
\intertext{We can compute that}
	\lambdanS &= \frac{\sum_{i\in[n]} d_i(d_i-1)}{2(\ell_n-1)},
	&
	\lambdanM &= \frac{\sum_{1\leq i<j\leq n} d_j(d_j-1)d_i(d_i-1)}{2(\ell_n-1)(\ell_n-3)}.
	}
The calculation of $ \lambdanS $ follows since the probability of a
connection between any two half-edges is $1/(\ell_n-1)$ and there are
$ \binom{d_i}{2} $ choices for the two half-edges that will form a
self-loop incident to the vertex $i$.  The calculation of $ \lambdanM $
follows since the probability for any two half-edges incident to the vertex
$ i $ to connect (in order) to any two half-edges incident to the vertex
$ j $ is $1/[(\ell_n-1)(\ell_n-3)]$.  Further, there are $\binom{d_i}{2} $
choices for the two-half edges incident to $ i $ and $\binom{d_j}{2} $
choices for the two-half edges incident to $ j $.  Finally, there are two
possible pairings of the two chosen half-edges incident to $ i $ to the two
chosen half-edges incident to $j$.
	
Throughout the paper, we write $f(n)=o(g(n))$ as $n\rightarrow \infty$ when $g(n)>0$ and $\lim_{n\rightarrow \infty}|f(n)|/g(n)=0$. We write $f(n)=O(g(n))$ as $n\rightarrow \infty$ when $g(n)>0$ and $\limsup_{n\rightarrow \infty}|f(n)|/g(n)<\infty$. Finally, we write $f(n)=\Theta(g(n))$ as $n\rightarrow \infty$ when $f(n)=O(g(n))$ and $g(n)=O(f(n))$.

In many cases and using this notation, we will approximate
	\eqan{
	\lambdanS &=(\nu_n/2)(1+O(1/n)),
	&
	\lambdanM &=(\nu_n^2/4) (1+O(1/n))-\chi_n, 
	}
where
	\eqan{
	\label{nun-def}
	\nu_n &=\frac{\sum_{i\in[n]} d_i(d_i-1)}{\ell_n},
	&
	\chi_n &=\frac{\sum_{i\in[n]} [d_i(d_i-1)]^2}{4(\ell_n-1)(\ell_n-3)}.
	}
For future purposes, we also define
	\eqn{
	\mu_n^{\sss(r)}=\frac{\sum_{i\in[n]} (d_i)_r}{\ell_n},
	}
where, for an integer $m$, we let $(m)_r=m(m-1)\cdots (m-r+1)$ denote the
$r$th factorial moment. (In particular, $\nu_n=\mu^{\sss(2)}_n$.)  We write
$\Law(X)$ for the distribution of $X$, and we write $(a \vee b)$ to denote
the maximum of $a$ and $b$. Our main result is as follows:

\begin{theorem}[Poisson approximation of self-loops and cycles]
\label{thm-PA-general}
For $\CMnd$, there exists a universal constant $C>0$ such that
	\eqn{
	\label{thm-1}
	\dtv{\Law(S_n)-\Po(\lambdanS)}\leq \frac{C}{(\nu_n/2\vee 1)}\frac{\nu_n^2}{\ell_n},
	}
	\eqn{
	\label{thm-2}
	\dtv{\Law(M_n)-\Po(\lambdanM)}\leq \frac{C}{(\lambdanM\vee 1)}\frac{(\mu_n^{\sss(3)})^2+\nu_n^4}{\ell_n},
	}
and
	\eqn{
	\label{thm-3}
	\dtv{\Law(S_n+M_n)-\Po(\lambdanS+\lambdanM)}\leq \frac{C}{((\lambdanS+\lambdanM)\vee 1)}\frac{(\mu_n^{\sss(3)})^2+\nu_n^4}{\ell_n}.
	}
In particular,
	\eqn{
	\label{thm-4}
	\prob(\CMnd\text{ \rm{simple}})=\prob(S_n+M_n=0)=\e^{-\lambdanS-\lambdanM}+r_n,
	\qquad
	\text{where}
	\qquad
	|r_n|\leq \frac{C}{((\lambdanS+\lambdanM)\vee 1)}\frac{(\mu_n^{\sss(3)})^2+\nu_n^4}{\ell_n}.
	}
\end{theorem}
\medskip

Let us discuss some of the history of this problem.  The configuration
model was introduced by Bollob\'as in \cite{Boll80b} to count the number of
regular graphs, and provides a very nice example of the {\em probabilistic
  method} (see also \cite{AloSpe00}).  Subsequently, the configuration
model has been used successfully to analyze many properties of random
regular graphs.  The number of simple graphs can be rather directly
obtained from the probability of simplicity of $\CMnd$ (see e.g.,
\cite[Proposition 7.6]{Hofs17}).  The introduction of the configuration
model was inspired by, and generalized the results in, the work of Bender
and Canfield \cite{BenCan78}.  See also Wormald \cite{Worm81a} and McKay
and Wormald \cite{MckWor91} for previous work. Before continuing with the history of Theorem \ref{thm-PA-general}, we discuss its implications on the number of
simple graphs with a prescribed degree sequence:

\begin{corollary}[Number of simple graphs with prescribed degrees]
\label{cor-PA-general}
The number $N_n(\bfdit)$ of simple graphs with degrees $\bfdit=(d_i)_{i\in[n]}$ satisfies
	\eqn{
	N_n(\bfdit)=\big(\e^{-\lambdanS-\lambdanM}+r_n\big) \frac{(\ell_n-1)!!}{\prod_{i\in[n]} d_i!},
	\qquad
	\text{where}
	\qquad
	|r_n|\leq \frac{C}{((\lambdanS+\lambdanM)\vee 1)}\frac{(\mu_n^{\sss(3)})^2+\nu_n^4}{\ell_n}.
	}
In particular, the number $N_n(r)$ of $r$-regular graphs with $n$ vertices satisfies, when $rn$ is even,
	\eqn{
	N_n(r)=\big(\e^{-(r-1)/2-(r-1)^2/4}+O(1/n)\big) \frac{(rn-1)!!}{(r!)^n}.
	}
\end{corollary}
      
The proof of Corollary \ref{cor-PA-general} follows directly from
\cite[Proposition 7.6]{Hofs17}, which implies that
	\eqn{
	\prob(\CMnd\text{ \rm{simple}})= N_n(\bfdit)\frac{\prod_{i\in[n]} d_i!}{(\ell_n-1)!!}.
	}

Let us continue the discussion of the history of the configuration model and Theorem \ref{thm-PA-general}. 
The configuration model, as well as uniform random graphs with a prescribed
degree sequence, were studied in greater
generality by Molloy and Reed in \cite{MolRee95, MolRee98}, where they
focus on the existence of a giant component.  The Poisson approximation for
the number of self-loops and multiple edges was first employed by
Bollob\'as \cite{Boll01} in the case of random regular graphs.  Janson
\cite{Jans06b} uses a Poisson approximation relying on the method of
moments for the number of vertices having self-loops and the pairs of
vertices having multiple edges between them.  He investigates the case where
the second moment of the degrees remains uniformly bounded, but not necessarily being uniformly integrable. In
\cite{Jans13a}, Janson revisits the problem for $S_n$ and $M_n$ in
\eqref{Sn-Mn-def} and uses the method of moments as well on the boundary
case where the maximal degree is of order $\sqrt{n}$.  Similar results were
proved previously in earlier versions of \cite{Hofs17}.  Janson's extension
in \cite{Jans13a} is inspired by the wish to deal with multiple edges and
self-loops for SIR epidemics on the configuration model in joint work with
Luczak and Windridge \cite{JanLucWin14}.

In contrast to the works above based on the moment method, we use a Poisson
approximation with couplings based on Stein's method, which also allows us to
give error estimates.  This method was recently used by Holmgren and Janson
\cite{HolJan14,HolJan15} to investigate the number of fringe trees in certain random trees. A major advantage is that Stein's 
method makes the approximation quantitative by giving explicit bounds on the error terms.
Contrary to Janson \cite{Jans06b,Jans13a}, our results do not allow for degrees that are of the order of $\sqrt{n}$.

\subsection{Regularity and moment assumptions on vertex degrees}
\label{sec-rvd}
\label{sec-fin-var}


We next investigate special cases of Theorem \ref{thm-PA-general}, under
stronger assumptions (on regularity and moments) on the degree distribution.

Let us now describe our regularity assumptions on the degree sequence
$\bfdit$ as $n\to\infty$. We denote the degree of a uniformly chosen vertex
$V$ in $[n]$ by $D_n=d_{\sss V}$.  The random variable $D_n$ has
distribution function $F_n$ given by
\eqn{
    \label{def-Fn-CM}
    F_n(x)=\frac{1}{n} \sum_{j\in [n]} \indic{d_j\leq x},
  }
where $\indicwo{A}$ denotes the indicator of the event $A$.    
Our regularity condition is as follows:

\begin{condition}[Regularity conditions for vertex degrees]
\label{cond-degrees-regcond}
~ The random variables $D_n$ converge in distribution to some random
variable $D$, and $\E[D_n] \to \E[D] < \infty$.
\end{condition}
  


Define 
	\eqn{
	\label{nu-def}
	\nu=\frac{\expec[D(D-1)]}{\expec[D]}.
	}
Under suitable assumptions on the second moment of $D_n$, we can deduce
more precise information about $S_n$ and $M_n$, and in particular consider
their joint distribution.  Our main results in the finite-variance case are
the following three theorems:

\begin{theorem}[Poisson approximation of self-loops and cycles]
\label{thm-PA-fin-var}
For $\CMnd$, where $\bfdit$ satisfies the Degree Regularity Condition
\ref{cond-degrees-regcond} and $\lim_{n\to\infty} \expec[D_n^2] =
\expec[D^2]<\infty$, it holds that
	\eqn{
	\dtv{\Law(S_n,M_n)-\Po(\nu/2)\otimes \Po(\nu^2/4)}\to 0.
	}
\end{theorem}

To prove Theorem \ref{thm-PA-fin-var}, we introduce a Cram\'er-Wold device for
Poisson random variables, that guarantees the independence of the limiting
random variables (see Section \ref{sec-CWD}), and that is of independent
interest.

Our methods also yield some speed of convergence results:

\begin{theorem}[Speed of convergence for self-loops and cycles under finite
  third moment]
\label{thm-PA-fin-var-b}
For $\CMnd$, where $\bfdit$ satisfies the Degree Regularity Condition \ref{cond-degrees-regcond} and $\lim_{n\rightarrow \infty}\expec[D_n^3]=\expec[D^3]<\infty$, it holds that
	\eqn{
	\label{Sn-Mn-Pois-a}
	\prob(S_n=0)=\e^{-\lambdanS}+O(1/n),
	\qquad
	\prob(M_n=0)=\e^{-\lambdanM}+O(1/n),
	}
and
	\eqn{
	\label{Sn-Mn-Pois-b}
	\prob(S_n=M_n=0)=\e^{-\lambdanS-\lambdanM}+O(1/n).
	}
In particular,
	\eqn{
	\label{Sn-Mn-Pois-c}
	\prob(\CMnd\text{ \rm{simple}})=\e^{-\lambdanS-\lambdanM}+O(1/n).
	}
Furthermore, when also $\lim_{n\rightarrow \infty}\expec[D_n^4]=\expec[D^4]<\infty$, $\lambdanS=\nu_n/2$ and $\lambdanM=\nu_n^2/4+O(1/n)$, so that \eqref{Sn-Mn-Pois-a}-\eqref{Sn-Mn-Pois-c} hold with
$\lambdanS$ and $\lambdanM$ replaced with $\nu_n/2$ and $\nu_n^2/4$.
\end{theorem}

\medskip

\begin{theorem}[Speed of convergence for self-loops and cycles with infinite third moment]
\label{thm-PA-fin-var-c}
For $\CMnd$, where $\bfdit$ satisfies the Degree Regularity Condition \ref{cond-degrees-regcond} and $\lim_{n\rightarrow \infty}\expec[D_n^2]=\expec[D^2]<\infty$, it holds that
	\eqn{
	\label{Sn-Mn-Pois-a2}
	\prob(S_n=0)=\e^{-\lambdanS}+O(1/n),
	\qquad
	\prob(M_n=0)=\e^{-\lambdanM}+O(\dmax^2/n),
	}
and
	\eqn{
	\label{Sn-Mn-Pois-b2}
	\prob(S_n=M_n=0)=\e^{-\lambdanS-\lambdanM}+O(\dmax^2/n).
	}
In particular,
	\eqn{
	\label{Sn-Mn-Pois-c2}
	\prob(\CMnd\text{ \rm{simple}})=\e^{-\lambdanS-\lambdanM}+O(\dmax^2/n).
	}
\end{theorem}
\medskip

Let us relate the above result to the scale-free behavior as observed in
many random graphs. See \cite[Chapter 1]{Hofs17} for an extensive
introduction to real-world networks and their power-law degree
sequences. In a power-law setting, we have that $[1-F_n](x)\approx c
x^{-(\tau-1)}$ (unless $x$ is too large).  Thus, the number of vertices of
degree at least $x$ equals
	\eqn{
	\label{power-law-max-degree}
	n[1-F_n](x)\approx c nx^{-(\tau-1)}.
	}
This is $\Theta(1)$ precisely when $x=\Theta(n^{1/(\tau-1)})$. Thus, one
can expect that $\dmax=\Theta(n^{1/(\tau-1)})$, so that the error terms in
\eqref{Sn-Mn-Pois-a2}--\eqref{Sn-Mn-Pois-c2} are of order
$n^{(3-\tau)/(\tau-1)}$, which is $o(1)$ when $\tau>3$. In turn, $\tau>3$
corresponds to $\expec[D^2]<\infty$, so that we are in the finite-variance
degree setting.
\medskip

In our proof, a concrete bound is given of the error term in the Poisson approximation in Theorem \ref{thm-PA-fin-var} in terms of the moments of $D_n$ and $n$, of the form as in \eqref{thm-1}--\eqref{thm-4}.

\subsection{Infinite-variance degrees}
\label{sec-inf-var}
In this section, we study the configuration model with infinite-variance degrees. We assume that $\nu_n\rightarrow \infty$. When the degrees obey a power-law with exponent $\tau\in(2,3)$, we will assume that
	\eqn{
	\label{nun-asymptotic}
	\nu_n=\Theta(n^{2/(\tau-1)-1}).
	}
This corresponds to a power-law degree distribution with exponent $\tau$ (recall \eqref{power-law-max-degree}), for which $d_{\rm max}=\max_{i\in[n]} d_i=\Theta(n^{1/(\tau-1)})$ with $\tau\in(2,3)$. In this case, our main result is the following central limit theorem for the number of self-loops:

\begin{theorem}[CLT for self-loops in $\CMnd$ with infinite-variance degrees]
\label{thm-PA-inf-var}
For $\CMnd$ where $\bfdit$ satisfies the Degree Regularity Condition \ref{cond-degrees-regcond}, while $\nu_n\rightarrow \infty$ as in \eqref{nun-asymptotic}. Then, for $\tau>2$,
	\eqn{
	\label{self-loops-nu-infty}
	\frac{S_n-\nu_n/2}{\sqrt{\nu_n/2}}\convd Z.
	}
\end{theorem}

Unfortunately, our proof does not apply to the multiple edges, since the number of multiple edges between vertices of degree $d_i\gg \sqrt{n}$ grows too rapidly. In this case, we need to truncate the degrees such that $\dmax=o(\sqrt{n})$:

\begin{theorem}[CLT for multiple edges in $\CMnd$ with infinite-variance degrees]
\label{thm-PA-inf-var-b}
Let $\dmax=o(\sqrt{n})$ and $\nu_n\rightarrow \infty$. Then,
	\eqn{
	\label{multiple-edges-nu-infty}
	\frac{M_n-\lambdanM}{\sqrt{\lambdanM}}\convd Z,
	}
where $Z$ is a standard normal random variable.
\end{theorem}
\medskip

Alternatively, we could also count only the multiple edges between vertices of degree $o(\sqrt{n})$. Indeed, take $m_n=o(\sqrt{n})$ and define
	\eqn{
	M_n^{\sss \rm(l)}=\sum_{1\leq i<j\leq n} \indic{d_i,d_j\leq m_n} X_{ij}(X_{ij}-1)/2 .
	}
Then, Theorem \ref{thm-PA-inf-var-b} also holds for $M_n^{\sss \rm(l)}$ with $\lambdanM$ replaced with $\expec[M_n^{\sss \rm(l)}]$ for any $m_n=o(\sqrt{n})$.

\subsection{Directed and bipartite configuration models}

In this section, we discuss the directed and bipartite configuration model. 

\paragraph{Self-loops and multiple edges in the directed configuration model.} For a general description of the directed configuration model we refer to Cooper and Frieze \cite{CooFri04} and van der Hofstad \cite[Section 7.8]{Hofs17}. Fix $\bfdit^{\sss \rm (in)}=(d_i^{\sss \rm (in)})_{i\in[n]}$ and $\bfdit^{\sss \rm (out)}=(d_i^{\sss \rm (out)})_{i\in[n]}$ to be sequences of in-degrees and out-degrees of the vertices $i\in[n]$, respectively. For a graph with in- and out-degree sequence $\bfdit=(\bfdit^{\sss \rm (in)},\bfdit^{\sss \rm (out)})$ to exist, we need to assume that
\eqn{\label{hej}
    	\hat{\ell}_n = \sum_{i\in [n]} d_i^{\sss \rm (in)}=\sum_{i\in [n]} d_i^{\sss \rm (out)}.
    	}
The directed configuration model $\DCMnd$ is obtained by pairing each in-half-edge to a uniformly chosen out-half-edge. Similarly as for the undirected case, we may investigate limit laws for the number of self-loops and multiple edges $(\hat{S}_n,\hat{M}_n)$. Self-loops occur if an in-half-edge is paired to an out-half-edge incident to the same vertex. Multiple edges occur between a pair of vertices, if two in-half-edges that are incident to one of the vertices are paired to two out-half-edges that are incident to the other vertex. Note that we are not considering two edges  with opposite directions between two vertices as a pair of multiple-edges (since this phenomenon actually often happens in real-world networks), but only  pairs of edges with the same direction. Thus, we define 
	\eqan{
	\label{Mnhat-def}
	\hat{M}_n&=\sum_{1\leq i<j\leq n} \Big[X_{ij}^{\sss \rm (in)}(X_{ij}^{\sss \rm (in)}-1)/2+X_{ij}^{\sss \rm (out)}(X_{ij}^{\sss \rm (out)}-1)/2\Big]\nn\\
	&=\sum_{i\neq j,~i,j\in[n]}X_{ij}^{\sss \rm (in)}(X_{ij}^{\sss \rm (in)}-1)/2=\sum_{i\neq j,~i,j\in[n]}X_{ij}^{\sss \rm (out)}(X_{ij}^{\sss \rm (out)}-1)/2,
	}
where $ X_{ij}^{\sss \rm (in)} $ are the number of edges between $ i $ and $ j $ that are directed from $ j $ to $ i $ and $ X_{ij}^{\sss \rm (out)} $ are the number of edges between $ i $ and $ j $ that are directed from $ i $ to $j$, and the last equality follows by the symmetry $X_{ij}^{\sss \rm (in)}=X_{ji}^{\sss \rm (out)}$.
\medskip

Let
	\eqn{
	\label{lambda-defsd}
	\hat{\lambda}_n^{\hS}=\expec[\hat{S}_n],\qquad \hat{\lambda}_n^{\hM}=\expec[\hat{M}_n].
	}
By similar calculations as in the undirected case we get 
	\eqn{
	\hat{\lambda}_n^{\hS}=\frac{\sum_{i\in[n]} d_i^{\sss \rm (in)}d_i^{\sss \rm (out)}}{\hat{\ell}_n},
	\qquad
	\hat{\lambda}_n^{\hM}=\frac{\sum_{i\neq j,~i,j\in[n]} d_i^{\sss \rm (in)}(d_i^{\sss \rm (in)}-1)d_j^{\sss \rm (out)}(d_j^{\sss \rm (out)}-1)}{2\hat{\ell}_n(\hat{\ell}_n-1)}.
	}
We also define 
	\eqn{
	\mu_n^{\sss(r,\rm{in})}=\frac{\sum_{i\in[n]} (d_i^{\sss \rm (in)})_r}{\hat{\ell}_n},
		\qquad
		\mu_n^{\sss(r,{\rm out})}=\frac{\sum_{i\in[n]} (d_i^{\sss \rm (out)})_r}{\hat{\ell}_n}.
	}
Then, our main result for the directed CM is as follows:
\begin{theorem}[Poisson approximation of self-loops and multiple edges in directed CM]
\label{thm-PA-dir}
For $\DCMnd$, there exists a universal constant $C>0$ such that
	\eqn{
	\label{thm-1d}
	\dtv{\Law(\hat{S}_n)-\Po(\hat{\lambda}_n^{\hS})}\leq \frac{C}{(\hat{\lambda}_n^{\hS}\vee 1)}\frac{( \hat{\lambda}_n^{\hS})^2}{\hat{\ell}_n},
	}
	\eqn{
	\label{thm-2d}
	\dtv{\Law(\hat{M}_n)-\Po(\hat{\lambda}_n^{\hM})}\leq \frac{C}
	{(\hat{\lambda}_n^{\hM} \vee 1)}\frac{\mu_n^{\sss(3,\rm{in})}\mu_n^{\sss(3,\rm{out})}+( \hat{\lambda}_n^{\hM})^2}{\hat{\ell}_n},
	}
and
	\eqn{
	\label{thm-3d}
	\dtv{\Law(\hat{S}_n+\hat{M}_n)-\Po(\hat{\lambda}_n^{\hS}+\hat{\lambda}_n^{\hM})}\leq \frac{C}{((\hat{\lambda}_n^{\hS}+\hat{\lambda}_n^{\hM})\vee 1)}\frac{\mu_n^{\sss(3,\rm{in})}\mu_n^{\sss(3,\rm{out})}+( \hat{\lambda}_n^{\hM})^2}{\hat{\ell}_n}.
	}
In particular,
	\begin{align}
	\label{thm-4d}
	&\prob\big(\DCMnd\text{ \rm{simple}}\big)=\prob(\hat{S}_n+\hat{M}_n=0)=\e^{-\hat{\lambda}_n^{\hS}-\hat{\lambda}_n^{\hM}}+\hat{r}_n,
	\nonumber\\&
	\text{where}
	\qquad
	|\hat{r}_n|\leq \frac{C}{((\hat{\lambda}_n^{\hS}+\hat{\lambda}_n^{\hM})\vee 1)}\frac{\mu_n^{\sss(3,\rm{in})}\mu_n^{\sss(3,\rm{out})}+( \hat{\lambda}_n^{\hM})^2}{\hat{\ell}_n}.
	\end{align}
\end{theorem}

\paragraph{Multiple edges in the bipartite configuration model.}
We continue with a discussion of multiple edges in the bipartite configuration model.
For a general description of the bipartite configuration model we refer e.g.\ to Blanchet and Stauffer \cite{BlaSta13} and Janson \cite{Jans13a}. Let $n^{\sss \rm (l)}$ denote the number of vertices on the left side of the bipartite graph, and $n^{\sss \rm (r)}$ the number of vertices on the right side of the bipartite graph. Fix 
 $\bfdit^{\sss\rm (l)}=(d_i^{\sss\rm (l)})_{i\in[n^{\sss \rm(l)}]}$ and $\bfdit^{\sss\rm (r)}=(d_j^{\sss\rm (r)})_{j\in[n^{\sss \rm(r)}]}$ degrees sequences for the two left and right parts, with  
	\eqn{
	\label{hej2}
    	\bar{\ell}_n = \sum_{i\in [n^{\sss \rm(l)}]} d_i^{\sss\rm (l)}=\sum_{j\in [n^{\sss \rm(r)}]} d_j^{\sss\rm (r)}.
    	}
The bipartite configuration model $\BCMnd$ is obtained by pairing each half-edge incident  to one of the vertices in $n^{\sss \rm(l)}$ to a uniformly chosen half-edge of those incident to the vertices in $n^{\sss \rm(r)}$. Thus, in this model there are obviously no self-loops.  However, there could exist multiple edges $\bar{M}_n$.  Multiple edges occur between a pair of vertices $(i,j)$ when two half-edges incident to a vertex $i\in[n^{\sss \rm(l)}]$ are paired to two half-edges  that are incident to a vertex $j\in[n^{\sss \rm(r)}]$.

To study the number of multiple edges, we define 
	\eqn{
	\label{Mnbar-def}
	\bar{M}_n=\sum_{i\in[n^{\sss \rm(l)}], j\in[n^{\sss \rm(r)}]}\bar{X}_{ij}(\bar{X}_{ij}-1)/2,
	}
where $ \bar{X}_{ij} $ denotes the number of edges between $ i $ and $ j $.
Let $\bar{\lambda}_n^{\hM}=\expec[\bar{M}_n]$. By similar calculations as for the standard configuration model we get
	\eqn{
	\bar{\lambda}_n^{\bM}=\frac{\sum_{i\in[n^{\sss \rm(l)}], j\in[n^{\sss \rm(r)}]} d_i^{\sss\rm (l)}(d_i^{\sss\rm (l)}-1)d_i^{\sss\rm (r)}(d_i^{\sss\rm (r)}-1)}{2\bar{\ell}_n(\bar{\ell}_n-1)}.
	}
We also define 
	\eqn{
	\mu_n^{\sss(k,\rm{l})}=\frac{\sum_{i\in[n^{\sss \rm(l)}]} (d_i^{\sss\rm (l)})_k}{\bar{\ell}_n},
		\qquad
		\mu_n^{\sss(k,\rm{r})}=\frac{\sum_{j\in[n^{\sss \rm(r)}]} (d_j^{\sss\rm (r)})_k}{\bar{\ell}_n}.
	}
Then, our main result for the bipartite CM is as follows:

\begin{theorem}[Poisson approximation of multiple edges in bipartite CM]
\label{thm-PA-bip}
For $\BCMnd$, there exists a universal constant $C>0$ such that
	\eqn{
	\label{thm-2b}
	\dtv{\Law(\bar{M}_n)-\Po(\bar{\lambda}_n^{\bM})}\leq \frac{C}{(\bar{\lambda}_n^{\bM} \vee 1)}\frac{\mu_n^{\sss(3,\rm{l})}\mu_n^{\sss(3,\rm{r})}+(\bar{\lambda}_n^{\bM})^2}{\bar{\ell}_n}.
	}
In particular,
	\eqn{
	\label{thm-4b}
	\prob(\BCMnd\text{ \rm{simple}})=\prob(\bar{M}_n=0)=\e^{-\bar{\lambda}_n^{\bM}}+\bar{r}_n,
	\qquad
	\text{where}
	\qquad
	|\bar{r}_n|\leq \frac{C}{(\bar{\lambda}_n^{\bM} \vee 1)}\frac{\mu_n^{\sss(3,\rm{l})}\mu_n^{\sss(3,\rm{r})}+(\bar{\lambda}_n^{\bM})^2}{\bar{\ell}_n}.
	}
\end{theorem}

\begin{rem}
For the directed configuration model $\DCMnd$ we can also prove results that correspond to Theorems \ref{thm-PA-fin-var}--\ref{thm-PA-inf-var-b} for the undirected case. 
Similarly, for the bipartite configuration model $\BCMnd$, we can prove results that correspond to Theorems \ref{thm-PA-fin-var-b}, \ref{thm-PA-fin-var-c} and \ref{thm-PA-inf-var-b} (recalling that there are no self-loops in this model). We leave these statements of the other results to the reader.
\end{rem}

\subsection{Discussion and open problems}
\label{sec-disc}
In this section, we discuss our results and provide open problems. 

Instead of investigating $S_n$ and $M_n$ as in \eqref{Sn-Mn-def}, one could also investigate other random variables that imply simplicity when the variable equals zero. An example would be to study $\widetilde M_n$ in \eqref{number-multiple-edges}. Another natural example would be
	\eqn{
	\label{SMgeq12}
	S_n^{\sss \rm(l)}=\sum_{i\in[n]} \indic{X_{ii}\geq 1},
	\qquad
	M_n^{\sss \rm(l)}=\sum_{1\leq i<j\leq n}\indic{X_{ij}\geq 2},
	}
as Janson does in \cite{Jans06b}. Both alternatives are of interest, since they all quantify different aspects of how many self-loops and multiple edges there are, and might satisfy central limits theorems for infinite-variance degrees for {\it different} values of $\tau$. However, application of Stein's method to these random variables is more difficult. For $M_n^{\sss \rm(l)}$, this is primarily due to the fact that the conditional distribution of $\indic{X_{kl}\geq 2}$ conditionally on $\indic{X_{ij}\geq 2}=1$ is quite involved.
\medskip

We next discuss configuration models in the power-law setting where $\tau\in(2,3)$ in some more detail. As we see in Theorems \ref{thm-PA-inf-var}--\ref{thm-PA-inf-var-b}, the number of self-loops and multiple edges in this case tend to infinity in probability, so that it is highly unlikely that there are none. This makes that the approach to obtain simple graphs by conditioning the configuration model to be simple is no viable option. However, real-world networks with power-law degrees with $\tau\in(2,3)$ are often observed, see e.g.\ the surveys in \cite{AlbBar02, Newm03a}. For example, Newman \cite{Newm03b}  proposes, amongst others, the configuration model with power-law degrees as a model for real-world networks, while Newman, Strogatz and Watts \cite{NewStrWat00} investigate the graph distances of such models. There is ample evidence that practitioners do wish to obtain simple graphs as a null-model for many real-world networks. There are many papers using the configuration model as null-models for real-world networks. In the case of infinite-variance degrees, this gives rise to an enormous problem. One possible solution to resolve this issue is to consider, instead, the {\em erased configuration model}, as suggested by Britton, Deijfen and Martin-L\"of \cite{BriDeiMar-Lof05}. In this erased model, self-loops are simply erased and multiple-edges merged to make the graph simple. While this does not produce a graph that has a uniform distribution over the space of all simple graphs, this model is highly practical, and we see that we only remove a small proportion of the edges so that the degree distribution is virtually unaltered (see e.g.\ \cite[Chapter 7]{Hofs17} for more details). This explains our interest in configuration models with power-law degrees. Theorem \ref{thm-PA-inf-var}--\ref{thm-PA-inf-var-b} can thus be seen as quantifications of the statement that we `do not remove many edges'. For example, Van der Hoorn and Litvak \cite{HooLit15} investigate the number of removed edges in the erased configuration model in the setting where the degrees are i.i.d.\ with distribution function $F$ satisfying $[1-F](x)=c x^{-(\tau-1)}$ for some $\tau\in (2,3)$  (in fact, they even allow for extra slowly-varying functions, but we refrain from discussing this generalization). The number of removed edges corresponds to 
	\eqn{
	R_n=S_n+\sum_{1\leq i<j\leq n} \big(X_{ij}-\indic{X_{ij}\geq 1}\big),
	}
which, for $\tau\in(2,3)$, is significantly smaller than $S_n+M_n$. 
\medskip

Gao and Wormald \cite{GaoWor14} take a different approach. Indeed, they investigate the number of simple graphs in the power-law case with $\tau\in(2,3)$. Under assumptions on $\mu_n^{\sss(r)}$, they investigate sharp asymptotics for $\prob(\CMnd\text{ simple})$. Let $M_r=\ell_n \mu_n^{\sss(r)}$, then \cite[Theorem 1]{GaoWor14} assumes that $M_2=o(M_1^{9/8})$. \cite[Theorem 1]{GaoWor14} assumes that the number of vertices of degree $k$ can be uniformly bounded by a constant times $nk^{-1/\tau}$ for $\tau>5/2$, which, in particular, implies that $\dmax=O(n^{1/\tau})$. These results are highly interesting, and show in particular that $\prob(S_n+M_n=0)=\e^{-(\nu_n/2+\nu_n^2/4)(1+o(1))}$ while at the same time giving an asymptotic expression of $o(1)$ in the exponent in terms of  $\mu_n^{\sss(r)}$ with $r=1,2$ and $3$. Such results cannot be obtained from ours, since the event of simplicity is an extreme value event when $\tau\in(2,3)$ with vanishing probability, whereas we study weak limits. On the other hand, the assumption that $\dmax=O(n^{1/\tau})=o(\sqrt{n})$ with $\tau\in(2,3)$ implies that we obtain a CLT for both the number of self-loops as well as the number of multiple edges by Theorem \ref{thm-PA-inf-var-b}.
Instead of assuming that $\dmax=O(n^{1/\tau})$, we prefer to work with cases where $\sum_{i\geq k} n_i=O(nk^{-1/(\tau-1)})$. This preference is inspired by the fact that the maximum of $n$ i.i.d.\ random variables with tail distribution function $1-F(x)=cx^{-(\tau-1)}$ is of order $n^{1/(\tau-1)}$ rather than $n^{1/\tau}$. In turn, a natural choice of deterministic power-law degrees arises when we take the number $n_k$ of vertices of degree $k$ to be equal to $n_k=\lceil nF(k)\rceil -\lceil nF(k-1)\rceil$, where again $1-F(x)=cx^{-(\tau-1)}$ is a power-law distribution and also $\dmax=O(n^{1/(\tau-1)})$. See also McKay and Wolmald \cite{MckWor91} for related work where sharp asymptotics were proved for the number of graphs with degrees that are potentially quite large. It is unclear to us whether the approach of Gao and Wormald can be extended to {\em simulate} random graphs with prescribed degree sequences in the regime where $\tau\in(2,3)$, as many practitioners would need.
\medskip

Another interesting problem is to investigate whether the CLT for the number of multiple edges $M_n$ can be extended to the full range $\tau\in(2,3)$ without the restriction that $\dmax=o(\sqrt{n})$. Our current proof relies on a Poisson approximation, which in particular can only be used when the mean and the variance of the asymptotic normal distribution are comparable. We believe that this is false for some $\tau\in(2,3)$. Instead, it would be interesting to investigate whether instead Stein's method for normal asymptotic distributions can be applied. This open problem is also interesting for other sums of indicators, for example for $\widetilde M_n$ in \eqref{number-multiple-edges} and for $M_n^{\sss \rm(l)}$ in \eqref{SMgeq12}.

\paragraph{Organization.} The remainder of this paper is organised as follows. In Section \ref{sec-prelim}, we present the preliminaries used in this paper, which include a novel Poisson Cram\'er-Wold device as well as bounds on Poisson approximations. In Section \ref{sec-couplings}, we present couplings of dependent indicators that will be crucial in applying Stein's method. In Section \ref{sec-proofs}, we present the proofs of our main results.

%
%
%
%
%

\section{Preliminaries}
\label{sec-prelim}

\subsection{Poisson Cram\'er-Wold device}
\label{sec-CWD}

In this section, we show that convergence of two random variables to two
\emph{independent} Poisson variables follows when we can prove convergence of
sums of their thinned versions.  This is to Poisson random variables, as
the Cram\'er-Wold device is for Normal variables.  We start by explaining
this method, which is of independent interest.  \medskip

Let $(X,Y)$ be two integer-valued random variables. Fix $p,q\in[0,1]$ and
define
	\eqn{
  	X_p = \Bin(X,p), \qquad\qquad Y_q = \Bin(Y,q)
	}
to be two binomial random variables, independent conditioned on $X,Y$. 
Then, the Poisson Cram\'er-Wold device is the following theorem:

\begin{theorem}[Poisson Cram\'er-Wold device]
\label{thm-PCW}
Suppose that, for every $p,q\in [0,1]$, $X_p+Y_q$ has a Poisson
distribution with mean $p\muX+q\muY$. 
Then $(X,Y)$ are two independent Poisson random variables with means $\muX$
and $\muY$, respectively.
\end{theorem}

\proof Let $M_{\sss X,Y}(s,t)=\expec[\e^{sX+tY}]$ denote the joint moment
generating function of a random vector $(X,Y)$ and
$M_{\sss X}(t)=\expec[\e^{tX}]$ the moment generating function of the
random variable $X$. Recall that the moment generating function of a
binomial random variable $X$ with parameters $n$ and $p$ equals
$M_{\sss X}(t)=(p\e^t+(1-p))^n$ and that of a Poisson random variable $Y$
with parameter $\lambda$ equals $M_{\sss Y}(t)=\e^{\lambda(\e^t-1)}$.
\medskip

We know that $M_{\sss X_p+Y_q}(t) =
\expec[\e^{t(X_p+Y_q)}]=\e^{(p\muX+q\muY)(\e^{t}-1)}$. We wish to show that
$M_{\sss X,Y}(t, s)=\e^{\muX(\e^{t}-1)+\muY(\e^{s}-1)}$, and it suffices to
prove this for $s,t\geq 0$.


We rewrite the moment generating function of $X_p+Y_q$ as
	\eqan{
	M_{\sss X_p+Y_q}(t)&=\expec[\e^{t(X_p+Y_q)}]
	=\expec[(p\e^{t}+(1-p))^{X}(q\e^{t}+(1-q))^{Y}]\\
	&=M_{\sss X,Y}(\log(p\e^{t}+(1-p)), \log(q\e^{t}+(1-q))).\nn
	}
Without loss of generality, we may assume that $t\geq s$.  We also assume
that $s\geq 0$.  We take $p=1$, so that $\log(p\e^{t}+(1-p))=t$, and $q$
such that $\log(q\e^{t}+(1-q))=s$.  Solving gives
$q=(\e^{s}-1)/(\e^{t}-1) \in [0,1]$, since $0\leq s\leq t$. Then we get that  
	\eqn{
	M_{\sss X,Y}(t, s)=\e^{\muX(\e^{t}-1)+q\muY(\e^{t}-1)}=\e^{\muX(\e^{t}-1)+\muY(\e^{s}-1)},
	}
as required. We conclude that $X$ and $Y$ are independent Poisson variables
with means $\muX$ and $\muY$, respectively.
\qed
\bigskip

\begin{corollary}[Poisson Cram\'er-Wold device for convergence]
\label{cor-PCW}
Let $X^{\sss(n)},Y^{\sss(n)}$ be sequences of random variables.  Suppose
that, for every $p,q\in [0,1]$, $X_p^{\sss(n)}+Y_q^{\sss(n)}$ converges in
distribution to a Poisson distribution with mean $p\muX+q\muY$.  Then
$(X^{\sss(n)},Y^{\sss(n)})$ converges to two independent Poisson random
variables with means $\muX$ and $\muY$, respectively.
\end{corollary}

This could be proved directly using characteristic functions in a similar
way as in Theorem \ref{thm-PCW} (where instead moment generating functions
were used).  Instead, we deduce this from Theorem \ref{thm-PCW}.

\proof Taking $p=1$ and $q=0$ we find that $X^{\sss(n)}$ is a tight
sequence, and similarly so is $Y^{\sss(n)}$, and thus there are
subsequences of $(X^{\sss(n)},Y^{\sss(n)})$ that converge in distribution.
Let $(X,Y)$ be some subsequential limit.  By  Theorem \ref{thm-PCW} we find
that $(X,Y)$ are independent Poisson variables as claimed.  Since every
subsequential limit has the same law, the convergence is along the entire
sequence. 
\qed \bigskip

\subsection{Poisson approximation}
\label{sec-PA}
We will make extensive use of Poisson approximations. For this, we rely on \cite[Theorem 2.C]{BarHolJan92}, which we quote for convenience. We start by introducing some notation. Let 
	\eqn{
	W=\sum_{\alpha\in \Lambda} I_{\alpha}
	}
be a sum of (possibly dependent) indicator functions indexed by some set
$\Lambda$. Let $(J_{\beta\alpha})_{\beta\in \Lambda\setminus \{\alpha\}}$
be a collection of indicator variables with the joint distribution of 
$((I_{\beta})_{\beta\in \Lambda}\mid I_{\alpha}=1)$, i.e., the conditional distribution of all other indicators {\em given} that $I_{\alpha}=1$. Let $p_{\alpha}=\prob(I_{\alpha}=1)$ and
	\eqn{
	\lambda=\sum_{\alpha\in \Lambda} p_{\alpha}=\expec[W].
	}
Note that the while the joint distribution of the variables
$(J_{\beta\alpha})$ is specified, the coupling with the family $(I_\alpha)$
can be chosen arbitrarily.  Then we have the following Poisson approximation:

\begin{theorem}[Poisson approximations \cite{BarHolJan92}]
\label{thm-PA-sums}
With $(J_{\beta\alpha})_{\beta\in \Lambda\setminus \{\alpha\}}$ as above,
	\eqn{
	\label{dtv-bd}
	\dtv{\Law(W),\Po(\lambda)}\leq (1\wedge \lambda^{-1}) \Big(\sum_{\alpha\in \Lambda}p_{\alpha}^2+\sum_{\alpha\in \Lambda}\sum_{\beta\neq \alpha} p_{\alpha} \expec[|I_{\beta}-J_{\beta\alpha}|]\Big).
	}
\end{theorem}
\medskip

As these are indicator variables, we can compute that
$\expec[|I_{\beta}-J_{\beta\alpha}|]=\prob(I_{\beta}\neq J_{\beta\alpha})$.
Our proof is based on finding an efficient coupling of $I_{\beta}$ and
$J_{\alpha\beta}$, i.e., one for which $J_{\beta\alpha}=I_\beta$ holds with high probability.


\section{Couplings for the number of self-loops and multiple edges in the CM}
\label{sec-couplings}

In this section, we investigate $(S_n,M_n)$ as defined in
\eqref{Sn-Mn-def}. We will rely on Poisson approximations as in
Theorem~\ref{thm-PA-sums}, for which it is convenient to rewrite
$(S_n,M_n)$ as
	\eqn{
	\label{Sn-Mn-def-rep}
	S_n=\sum_{i\in[n]} \sum_{1\leq s< t\leq d_i} L_{st},
	\qquad\qquad
	M_n=\sum_{1\leq i<j\leq n}\sum_{1\leq s_1< s_2\leq d_i} \sum_{1\leq t_1\neq t_2\leq d_j} L_{s_1t_1,s_2t_2}.
	}
Here $L_{st}$ is the indicator that the half-edges $s$ and $t$ that are
incident to the same vertex are paired to form a self-loop, while
$L_{s_1t_1,s_2t_2}$ is the indicator that $s_1t_1$ and $s_2t_2$ are paired
together, where $s_1,s_2$ are incident to the same vertex, as are
$t_1,t_2$.  Note that we may assume $s_1<s_2$, but swapping $t_1,t_2$ will
lead to different configurations, so their order is not given.

We use the Poisson approximation in Theorem \ref{thm-PA-sums}, jointly with the Poisson Cram\'er-Wold device in Theorem \ref{thm-PCW} and thus deal with
	\eqn{
	\label{CW-variable}
	W=\sum_{i\in[n]} \sum_{1\leq s< t\leq d_i} L_{st}K_{st}
	+\sum_{1\leq i<j\leq n}\sum_{1\leq s_1< s_2\leq d_i} \sum_{1\leq t_1\neq t_2\leq d_j}L_{s_1t_1,s_2t_2}K_{s_1t_1,s_2t_2},
	}
where $(K_{st})$ are i.i.d.\ Bernoulli's with probability $p$ and
$(K_{s_1t_1,s_2t_2})$ are i.i.d.\ Bernoulli's with probability $q$. We need
to describe the law of the indicators conditioned on $I_{\alpha}=1$. Here
$\alpha$ can be $st$ or $s_1t_1,s_2t_2$ and
$I_{\alpha}=L_{\alpha}K_{\alpha}$. Note that the $K_{\alpha}$'s are
completely independent of everything else, so they do not change the
story.  For simplicity, we will just deal with $K_{\alpha}\equiv 1$, though
all computations below are valid for any set of $K$'s.

\paragraph{The success probabilities.}
We start by analyzing the ``success'' probabilities $p_{\alpha}$. For
$\alpha$ corresponding to a self-loop,
	\eqn{
	p_{\alpha}=\frac{1}{\ell_n-1},
	\qquad
	\text{where}
	\qquad
	\ell_n=\sum_{i\in[n]} d_i.
	}
For $\alpha$ corresponding to a pair of edges, 
	\eqn{
	p_{\alpha}=\frac{1}{(\ell_n-1)(\ell_n-3)}.
	}
This gives us that
	\eqan{
	\lambda&=\sum_{i\in[n]} \sum_{1\leq s\leq t\leq d_i}\frac{1}{\ell_n-1}+\sum_{1\leq i<j\leq n}\sum_{1\leq s_1\leq s_2\leq d_i} \sum_{1\leq t_1\neq t_2\leq d_j}\frac{1}{(\ell_n-1)(\ell_n-3)}\nn\\
	&=\frac{1}{2(\ell_n-1)}\sum_{i\in[n]} d_i(d_i-1)+ \frac{1}{4(\ell_n-1)(\ell_n-3)}\sum_{i\neq j \in[n]} d_i(d_i-1)d_j(d_j-1)\nn\\
	&=[\nu_n/2+\nu_n^2/4](1+o(1)).
	}
Further,
	\eqn{
          \sum_{\alpha}p_{\alpha}^2 =
          \left[\frac{\nu_n}{2\ell_n}+\frac{\nu_n^2}{4\ell_n^2}\right](1+o(1))
        = o(1).
	}
as long as $\dmax=o(n)$.  This bounds the easier term $\sum p_\alpha^2$ on
the right-hand side of \eqref{dtv-bd}.  We now turn to the more involved
contribution in the right-hand side of \eqref{dtv-bd}, for which the task
is to give a convenient and efficient description of the distribution of
$(J_{\beta\alpha})_{\beta}=((I_{\beta})_{\beta\in \Lambda}\mid
I_{\alpha}=1)$.  This means that we need to study the 
distribution of $L_{\alpha}$ {\em conditionally} on $L_{\beta}=1$.  More
precisely, below we define a coupling of the $L_\alpha$
and the conditioned $L_\alpha$ for each $\beta$.


Before describing the coupling in the different cases that can arise, we
make the following observation.  For some pairs $\alpha$,$\beta$ it is the
case that $L_\alpha=1$ and $L_\beta=1$ are incompatible.  This happens
whenever these events require the same half-edge to be matched in two
different ways.  In that case, conditioning on $L_\alpha=1$ makes
$L_\beta=0$.  In all other cases, it is easy to see that $L_\alpha$ and
$L_\beta$ are positively correlated.  For example, if both $\alpha,\beta$
are self-loop events, then $\E[L_\alpha L_\beta] =
\frac{1}{(\ell_n-1)(\ell_n-3)} > \frac{1}{(\ell_n-1)^2}$.  Similar
inequalities hold when one or both of the two are multiple edge events.
There are also pairs $\alpha,\beta$ which include a common matched pair, in
which case the correlation is very strong.
In light of this positive correlation, for any compatible pair
$\alpha,\beta$, the conditioned $L_\beta$ stochastically dominates the
unconditioned $L_\beta$.  Our coupling realizes this stochastic
domination: If $\alpha,\beta$ are compatible, then forcing $L_\alpha$ can
only increase $L_\beta$.

Since $\alpha$ and $\beta$ can be of two distinct types, corresponding to
self-loops and multiple edges, this gives rise to four different cases.  We
start with the conditional law of $L_{s't'}$ conditionally on $L_{st}=1$:

  
\subsection*{(a) Conditional law of $L_{s't'}$ conditionally on
  $L_{st}=1$.}
To create the conditional law of $(J_{\beta\alpha})$, which is the same as
the joint law of $(L_{\beta})$ given $L_{\alpha}=1$ with $\alpha=st$, we
start with $\CMnd$, giving us the unconditional law of $(L_{\beta})$. When
the half-edges $s$ and $t$ have been paired to one another, we do nothing,
because then $L_{st}=1$ already. When $L_{st}=0$, we break open the two
edges containing $s$ and $t$ respectively, pair $s$ and $t$, and pair the
two other half-edges that are now unpaired to each other.  We refer to this
as \emph{rewiring to create $\alpha$}. It is clear that
the resulting graph is $\CMnd$ conditioned on half-edges $s,t$ being
matched, and thus produces the required distribution $(J_{\beta\alpha})$,
and also couples it with the unconditioned law $(L_{\beta})$.  We now
compute $\expec[|L_{\beta}-J_{\beta\alpha}|] = \prob(L_{\beta} \neq
J_{\beta\alpha})$ (and in this first case we assume that $\beta=s't'$
corresponds to a self-loop).

We note that $J_{\beta\alpha}=L_{\beta}$, unless the self-loop $\beta$ is
present and is destroyed, or the self-loop $\beta$ is absent and is
created.  We have two different cases depending on whether $\alpha$ and
$\beta$ are incident to two distinct vertices or they are incident to the
same vertex. We will now examine the contributions to each of these two
cases.

\paragraph{Case (a1): The self-loops $\alpha$ and  $\beta$ are incident to
  the distinct vertices $i$ and $j$:}
We start with the case where $\alpha=st$ and $\beta=s't'$ are incident to
two distinct vertices $i$ and $j$.  We first note that rewiring can never
destroy the self-loop $\beta$, since the half-edges $s'$ and $t'$ that are
incident to the vertex $j$ can only be affected if before the rewiring they
were paired to the half-edges $s$ and $t$ that are incident to the vertex
$i$.  From this fact it also follows that $\beta$ is created exactly if
before the rewiring, the half-edges $s'$ and $t'$ in $\beta$ are paired (in
some order) to the half-edges $s$ and $t$ in $\alpha$.  This has
probability $\frac{2}{(\ell_n-1)(\ell_n-3)}$.  Note that for two distinct
vertices $i$ and $j$, there are $\binom{d_i}{2}\binom{d_j}{2}$ choices for
the pair $(s,t)$ incident to $i$ and the pair $(s',t')$ incident to $j$.


\paragraph{Case (a2): The self-loops $\alpha$ and $\beta$ are incident
  to the same vertex $i$, and are disjoint:} 
We now consider the case where $\alpha=st$ and $\beta=s't'$ are incident to
the same vertex $i$, and do not share any half-edge, so that
$\{s,t\}\cap \{s',t'\}=\varnothing$.  As in case (a1), rewiring can not
destroy the self-loop $s't'$, and creates the self-loop $s't'$ precisely
when $s,t$ are paired to $\{s',t'\}$ in some order.  This occurs with
probability $\frac{2}{(\ell_n-1)(\ell_n-3)}$.  Note that there are
$6\binom{d_i}{4}$ choices for half-edges $s,t$ and half-edges $s',t'$ incident to 
the vertex $i$.

\paragraph{Case (a3): The self-loops $\alpha$ and $\beta$ are incident
  to the same vertex $i$, and overlap.} 
Consider finally the case of self loops with
$\{s,t\}\cap \{s',t'\}\neq \varnothing$.  If $s't'$ is a self-loop before
rewiring, it is destroyed by the rewiring. This occurs with probability
$\frac{1}{(\ell_n-1)}$.  Note that there are $6\binom{d_i}{3}$ choices for
the pairs $\{s,t\}$ and $\{s',t'\}$ with an overlap.

\bigskip

Recall the notation $(m)_k=m(m-1)\cdots(m-k+1)$.
Using that $p_{\alpha}=1/(\ell_n-1)$, the total contribution from cases
(a1)--(a3) to the second sum in \eqref{dtv-bd} is thus equal to
	\eqan{\label{case(a)}
  	&\sum_{i\neq j\in[n]} \frac{(d_i)_2(d_j)_2}{4} \cdot
  	\frac{2}{(\ell_n-1)^2(\ell_n-3)}
  	+\sum_{i\in[n]} \frac{(d_i)_4}{4}\cdot \frac{2}{(\ell_n-1)^2(\ell_n-3)}
  	+ \sum_{i\in[n]} (d_i)_3 \cdot \frac{1}{(\ell_n-1)^2} \\
  	&\qquad\qquad\qquad= O(\nu_n^2/\ell_n) + O(\mu_n^{\sss(3)}/\ell_n),\nn
	}
where the first sum corresponds to the total contribution for the case when
$\alpha$ and $\beta$ are incident to two distinct vertices $i$ and $j$, and
the last two sums correspond to the total contribution for the case when
$\alpha$ and $\beta$ are incident to the same vertex $i$.

\subsection*{(b) Conditional law of $L_{s_1't_1', s_2't_2'}$ conditionally
  on $L_{st}=1$.}
Continuing our analysis of the above coupling, we now consider
$\beta=\{s_1't_1', s_2't_2'\}$ corresponding to a pair of parallel edges.
We have different cases depending on whether the half-edges in $\alpha$ and
$\beta$ are incident to three distinct vertices or only two different
vertices.  We will now examine the contributions to each of these two
cases.

\paragraph{Case (b1): The self-loop $\alpha$ and the multiple-edge $\beta$ are incident to three vertices:}
We start with the case where $\alpha=st$ is incident to a vertex $i$, and
$\beta=\{s_1't_1', s_2't_2'\}$ is incident to two other vertices, so that
$s_1',s_2'$ are incident to vertex $j$ and the pair $t_1', t_2'$ are
incident to vertex $k$, with $\{i,j,k\}$ all distinct.  Note that (as in
case (a1)), rewiring can not destroy the multiple edge $\beta$, since
$\beta$ is only affected if there is some half-edge in $\beta$ that is
paired to some half-edge in $\alpha$ before the rewiring, in that case
$\beta$ could not have been present in $\CMnd$.  However, rewiring can
create the multiple edge $\beta$.  This occurs when before the rewiring
either the edge $s_1't_1'$ or the edge $s_2't_2'$ already existed and the
half-edges $s$ and $t$ are paired to the two remaining half-edges from
$\beta$; see Figure \ref{rewire1} for an illustration.
\begin{figure}
\includegraphics{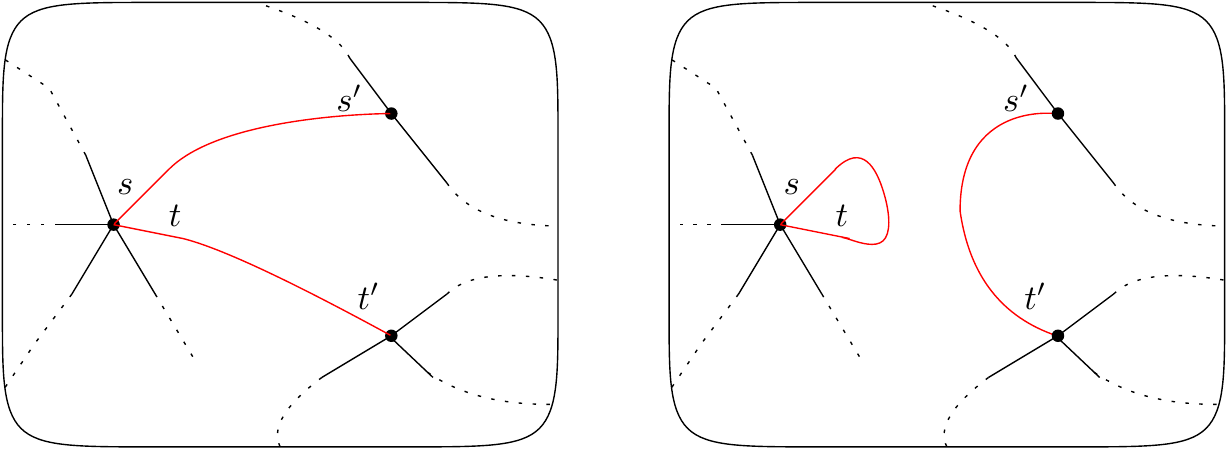}
\caption{Rewiring the left configuration so as to create the self-loop 
with half-edges $\{s,t\}$ results in creation of the edge $\{s',t'\}$. 
If a second edge $\{s'_2,t'_2\}$ is already present (not shown) then a 
multiple edge will be formed.}\label{rewire1}
\end{figure} 

  We thus have four symmetric cases: One of these cases is when $s$
was paired to $s_1'$ and $t$ to $t_1'$, while $s_2'$ was paired to $t_2'$.
Thus, the total probability for these four symmetric cases is
$\frac{4}{(\ell_n-1)(\ell_n-3)(\ell_n-5)}$.  Note that there are
$\binom{d_i}{2}$ choices for the pair of half-edges $(s,t)$ in the vertex
$i$ and then $2\binom{d_j}{2}\binom{d_k}{2}$ to choose the multiple edge
$\beta=\{s_1't_1', s_2't_2'\}$ between vertices $j$ and $k$.  When we sum
over all vertices $i,j$ and $ k$, we could either assume that $j<k$ or
divide the total sum by 2, since we can permute $j$ and $k$. In total, using that $p_{\alpha}=1/(\ell_n-1)$,  the contribution
from Case (b1)	to the second sum in \eqref{dtv-bd} is
	\eqn{
	\sum_{i\neq j\neq k}\frac{(d_i)_2(d_j)_2(d_k)_2}{8}\cdot \frac{4}{(\ell_n-1)^2(\ell_n-3)(\ell_n-5)}=\frac{6(d_i)_2(d_j)_2(d_k)_2}{2(\ell_n-1)^2(\ell_n-3)(\ell_n-5)}.
      	}
Here, the $6$ in the first term comes from the possible orders of $i,j,k$.
      
\paragraph{Case (b2): The self-loop $\alpha$ and the multiple-edge $\beta$
  are incident to two vertices:}
We now consider the case when $\alpha=st$ and
$\beta=\{s_1't_1', s_2't_2'\}$ are incident to only two distinct vertices.
Specifically, we assume that both the half-edges $s,t$ and $s_1',s_2'$ are
all incident to the vertex $i$, while the half-edges $t_1',t_2'$ are
incident to a different vertex $j$.  This is split into three sub-cases,
depending on whether $\{s,t\}$ and $\{s'_1,s'_2\}$ have zero, one, or two
elements in common.

If $\{s_1',s_2'\}\cap \{s,t\} = \varnothing$ then we can not destroy
$\beta$ since no half-edge in $\beta$ could have been paired to $s$ or $t$
before the rewiring. However, we can create $\beta$. This again occurs when
before the rewiring either the edge $s_1't_1'$ or the edge $s_2't_2'$
already existed and the half-edges $s$ and $t$ are paired to the two
remaining half-edges in $\beta$.  The total probability for these four
symmetric cases is the same as in case (b1):
$\frac{4}{(\ell_n-1)(\ell_n-3)(\ell_n-5)}$.  Note that there are
$2\binom{d_i}{2} \binom{d_j}{2}\binom{d_i-2}{2}$ choices for the multiple
edge $\beta$ incident to the vertices $i$ and $j$ and the self-loop
$\alpha$ incident to vertex $i$.

If $\{s_1',s_2'\}\cap \{s,t\} \neq \varnothing$ then rewiring can not
create $\beta$, since one of the half-edges in $\beta$ must be part of the
self-loop $\alpha$ after the rewiring.  In case $\{s_1',s_2',s,t\}$ are
three distinct half-edges so that $s=s_1'$ or $s=s_2'$, then we destroy
$\beta$ if $\beta$ existed before the rewiring (while the final half-edge
$t$ incident to vertex $i$ was paired to an arbitrary half-edge).  These two
symmetric cases thus have probability
$\frac{2}{(\ell_n-1)(\ell_n-3)}$, and there are
$ 2\binom{d_i}{2} \binom{d_j}{2}(d_i-2)$ choices for the multiple
edge $\beta$ incident to the vertices $i$ and $j$ and the remaining
half-edge $t$ in the self-loop $ \alpha $ incident to vertex $i$.

Finally we consider the case when $s=s_1'$ and $t=s_2'$.  Then we destroy
$\beta$ if $\beta$ existed before the rewiring, which occurs with
probability $\frac{1}{(\ell_n-1)(\ell_n-3)}$. Note that there are
$ 2\binom{d_i}{2} \binom{d_j}{2}$ choices for the multiple edge $\beta$
incident to the vertices $i$ and $j$ and then the self-loop $ \alpha $
incident to vertex $i$ is also decided from that choice.

\bigskip

In total, again using $p_{\alpha}=1/(\ell_n-1)$,  the contribution
from Cases (b1) and (b2) to the second sum in \eqref{dtv-bd} is
	\eqan{\label{case(b)}
  	&\sum_{i<j<k\in[n]}
  	\frac{6(d_i)_2(d_j)_2(d_k)_2}{2(\ell_n-1)^2(\ell_n-3)(\ell_n-5)}+\sum_{i\neq j\in[n]}
  	\frac{(d_i)_4(d_j)_2}{(\ell_n-1)^2(\ell_n-3)(\ell_n-5)}+\sum_{i\neq
    	j\in[n]} \frac{2(d_i)_3(d_j)_2+(d_i)_2(d_j)_2}{2(\ell_n-1)^2(\ell_n-3)}
  	\\
  	&\qquad\qquad\qquad= O(\nu_n^3/\ell_n)+ O(\mu_n^{\sss(4)}\nu_n/\ell_n^2)+O(\mu_n^{\sss(3)}\nu_n/\ell_n).\nn
	}

Note that in this paper we do not consider the joint distribution of $S_n$
and $M_n$ when $\nu_n\rightarrow\infty$, so this term will only be used for
$\nu_n=O(1)$.

\subsection*{(c) Conditional law of $L_{s't'}$ conditionally on
  $L_{s_1t_1,s_2t_2}=1$.}

To deal with Case (c), we rely on symmetry that is present in our setting.
The simple observation in the lemma below is described in \cite[p.25]{BarHolJan92}, but we prove it for completeness.
\begin{lemma}[Symmetry] With the notation in Section \ref{sec-PA},
  \[
    p_\alpha \E[I_\beta-J_{\beta\alpha}] = p_\beta
    \E[I_\alpha-J_{\alpha\beta}] = -\mathrm{Cov}(I_\alpha,I_\beta).
  \]
\end{lemma}

\begin{proof}
  We have that $p_\alpha\E[I_\beta] = p_\alpha p_\beta$, and
    	$$
	p_\alpha \E[J_{\beta\alpha}] = \prob(I_\alpha=I_\beta=1)=\E[I_\alpha I_\beta].
	$$  
Thus, the difference is the covariance (multiplied by -1) and is invariant to swapping $\alpha$ and $\beta$,
\end{proof}

In our setting, for compatible $\alpha$ and $\beta$ the difference
$I_\beta-J_{\beta\alpha}$ is never positive, while for incompatible
$\alpha,\beta$ it is never negative.  Thus, 
	\[
	p_\alpha \big|\E[I_\beta-J_{\beta\alpha}]\big|=p_\alpha \E[|I_\beta-J_{\beta\alpha}|].
	\]
We conclude that $p_\alpha \E[|I_\beta-J_{\beta\alpha}|]$ is also invariant to swapping $\alpha$ and $\beta$.  In particular the sum
over self-loops $\alpha$ and multiple edges $\beta$ is the same as the sum over multiple edges $\alpha$ and self-loops $\beta$, and thus the contribution from Case (c) is {\it equal} to the contribution from Case (b).

\subsection*{(d) Conditional law of $L_{s_1't_1', s_2't_2'}$ conditionally
  on $L_{s_1t_1,s_2t_2}=1$.}
We now turn our coupling to the case where $\alpha=\{s_1t_1,s_2t_2\}$ is a pair of
parallel edges.  We rewire $\CMnd$ to create the coupled variables
$(J_{\beta\alpha})$, with the joint law of $(L_{\beta})$ given
$L_{\alpha}=1$.  Start with $\CMnd$, giving us the unconditioned
$(L_{\beta})$.  If the pairs of half-edges $(s_1,t_1)$ and $(s_2,t_2)$ are
already paired, then $L_{s_1t_1,s_2t_2}=1$ already, and there is nothing to
be done.

When $L_{s_1t_1,s_2t_2}=0$, we break open all the edges containing
$s_1,s_2,t_1,t_2$.  This leaves these and at most four additional half-edges
unmatched.  We then pair $s_1$ to $t_1$ and pair $s_2$ to $t_2$.  The
additional unmatched half-edges (of which there are zero, two, or four) are paired randomly.
This produces $(J_{\beta\alpha})$ with the needed distribution, coupled 
with the original $(L_{\beta})$.  We shall now estimate
$\expec[|L_{\beta}-J_{\beta\alpha}|]$.  We note that
$J_{\beta\alpha}=L_{\beta}$, unless the multiple-edge $\beta$ is present
and is destroyed, or the multiple-edge $\beta$ is absent and is created.
Note that we have several cases depending on how the multiple-edges
$\alpha$ and $\beta$ intersect, and whether they are incident to two, three
or four distinct vertices.

\paragraph{Case (d1): The multiple edges $\alpha$ and $\beta$ are incident
  to four distinct vertices:} 
We start with the case when $\alpha$ and $\beta$ are incident to four
different vertices, $\alpha$ to $i,j$, and $\beta$ to $k,l$.  Note that in
this case rewiring can not destroy the multiple-edge $\beta$, since if
$\beta$ existed then the half-edges in $\beta$ could not have been
paired to the half-edges in $\alpha$ before the rewiring.
However, rewiring can
create $\beta$.  This can happen in two different ways.  In the first way,
all half-edges in $\alpha$ are paired to all half-edges in $ \beta $,
which just has the probability
	\begin{align}\label{P1}
	P_1=\frac{4!}{(\ell_n-1)(\ell_n-3)(\ell_n-5)(\ell_n-7)}.
	\end{align} 
The second way $\beta$ can be created is if one of the edges of $\beta$ was
present before rewiring, while the remaining two half-edges in $\beta$ were
paired to two-half-edges in $\alpha$ (the remaining two half-edges in
$\alpha$ can be paired arbitrarily). 
 
This has probability
	\begin{align}\label{P2}
	P_2=\frac{2\cdot 4\cdot 3}{(\ell_n-1)(\ell_n-3)(\ell_n-5)}.
	\end{align}

Note that there are
\[
  4\binom{d_i}{2}\binom{d_j}{2}\binom{d_k}{2}\binom{d_l}{2}
\]
ways to choose $\alpha=\{s_1t_1, s_2t_2\}$ incident to vertices $i$ and $j$,
and $\beta=\{s_1't_1', s_2't_2'\}$ incident to vertices $k$ and $l$.  (Also
note that in both of the cases just described, after the rewiring we could
possibly have created $\beta$, but in neither of the cases it is certain
that $\beta$ has been created.) See Figure \ref{rewire2} for an illustration of these two possibilities for rewiring edges.
	\begin{figure}
\includegraphics{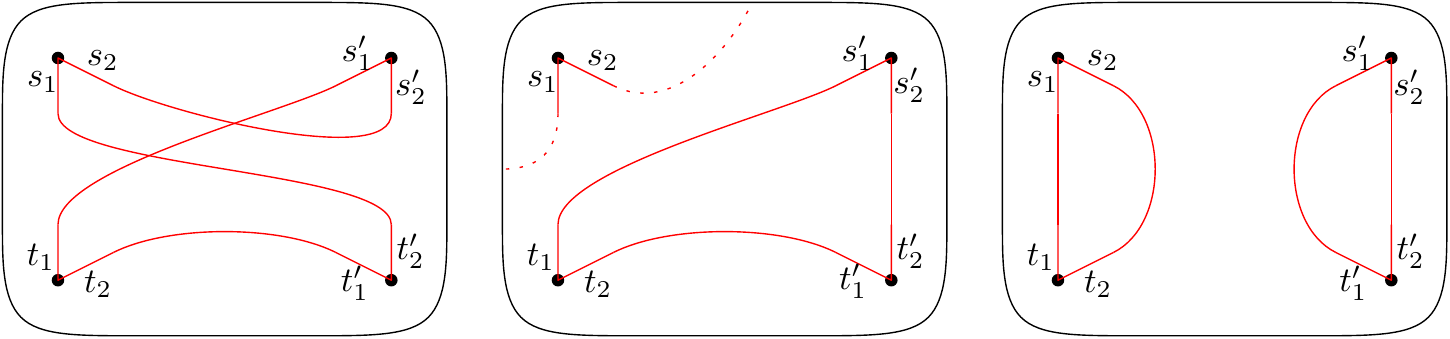}
\caption{In the left and centre configurations, rewiring to create the 
parallel edges $\{s_1t_1,s_2t_2\}$ between the two left vertices will 
create (with positive probability) the parallel edges 
$\{s'_1t'_1,s'_2t'_2\}$ between the right vertices, as shown on the 
right. Other half-edges on the same vertices are not shown.}\label{rewire2}\end{figure}

 When we sum over all vertices $i,j,k$ and
$l$ we have to divide the total sum by 4 similarly as we divided by 2 in
the previous Case (b) when there were three vertices that were
incident to $\alpha$ and $\beta$.

Using that $p_{\alpha}=\frac{1}{(\ell_n-1)(\ell_n-3)}$, with a factor of
$24$ for permuting $i,j,k,l$, we find that the total contribution to the
second sum in \eqref{dtv-bd} due to case (d1) is bounded by
\eqn{
  \label{case(dcase1)}
	24\sum_{i<j<k<l \in[n]} \frac{(d_i)_2(d_j)_2(d_k)_2(d_l)_2}{16}
        p_\alpha (P_1+P_2) = O(\nu_n^4/\ell_n).
	}

\paragraph{Case (d2): The multiple edges $\alpha$ and $\beta$ are incident
  to three distinct vertices:} 
We continue with the case when $\alpha$ and $\beta$ are incident to only
three different vertices $i,j$ and $i,k$.  We can assume that $s_1,s_2$ and
$s_1',s_2'$ are incident to vertex $i$, that $t_1,t_2$ are incident to $j$
and $t_1',t_2'$ are incident to $k$.  There are sub-cases, according to the
number of common half-edges among $\alpha$ and $\beta$.

When $\{s_1, s_2\} \cap \{s_1',s_2'\} = \varnothing$, we can not destroy the
multiple-edge $\beta$, since we have eight different half-edges in $\alpha$
and $ \beta$. In this case we can again create $\beta$ if the half-edges are paired as
described in the previous case i.e., with probability $P_1$ in (\ref{P1})
and probability $P_2$ in (\ref{P2}) respectively (there is a possibility
that $\beta$ is created).  Note that there
are 
	$$
	4\binom{d_i}{2}\binom{d_i-2}{2}\binom{d_j}{2}\binom{d_k}{2}
	$$ 
ways to
choose $\alpha=\{s_1t_1, s_2t_2\}$ incident to vertices $i$ and $j$ and
$\beta=\{s_1't_1', s_2't_2'\}$ incident to vertices $i$ and $k$.  If
$\{s_1,s_2\} \cap \{s_1', s_2'\} \neq \varnothing$, then we can not create
the multiple edge $ \beta $ incident to the vertices $ i $ and $ k $ since
after the rewiring at least one of the half-edges $ s_1',s_2' $ in
$ \beta $ is paired to a half-edge in $\alpha$ that is incident to the
vertex $ j$. However, when $ \beta $ existed before the rewiring, it is
destroyed by the same reason.  Hence, we have two possibilities for this to
happen i.e., $|\{s_1,s_2\} \cap \{s_1', s_2'\}|$ is equal to 1 or 2. The multiple edge
$ \beta $ exists with probability
$ \frac{1}{(\ell_n-1)(\ell_n-3)} $. Note that in the case when $\alpha$
and $ \beta $ contain three distinct half-edges incident to $i$ and two
distinct half-edges incident to $j$ and $k$, respectively, we have
	$$
	8\binom{d_i}{2}(d_i-2)\binom{d_j}{2}\binom{d_k}{2}
	$$ 
ways to choose $\alpha$ and $\beta$, whereas in the case when  $\alpha$ and $ \beta $ contain two distinct half-edges incident to $i$, $j$ and $k$, respectively, we have 
	$$
	4\binom{d_i}{2}\binom{d_j}{2}\binom{d_k}{2}
	$$ 
ways to choose $\alpha$ and $\beta$.
Using that $p_{\alpha}=\frac{1}{(\ell_n-1)(\ell_n-3)}$, the total contribution to the second sum in \eqref{dtv-bd} due to case (d2)
is thus bounded by
	\eqan{
	\label{case(dcase2new)}
  	&\sum_{i\neq j\neq k \in[n]} \frac{(d_i)_4(d_j)_2(d_k)_2}{4} p_\alpha
  	(P_1+P_2)
  	+ \sum_{i\neq j\neq k \in[n]} \dfrac{2(d_i)_3(d_j)_2(d_k)_2+
    	(d_i)_2(d_j)_2(d_k)_2}{(\ell_n-1)^2(\ell_n-3)^2} \\
  	&\qquad\qquad\qquad= O(\mu_n^{\sss(4)}\nu_n^2/\ell_n^2)+O(\mu_n^{\sss(3)}\nu_n^2/\ell_n)+O(\nu_n^3/\ell_n).\nn
	}

\paragraph{Case (d3): The multiple edges $\alpha$ and $\beta$ are
  compatible and incident to two vertices:}
We finally consider the case when $\alpha$ and $\beta$ are incident to two
different vertices $i$ and $j$.  In this case rewiring can both create and
destroy $\beta$ when forcing $\alpha$.  In the case when all the eight
half-edges are distinct we can again not destroy $\beta$.  As in case (d1)
and (d2) when there were eight distinct half-edges in $\alpha$ and $\beta$
and as described above there were two different scenarios when $\beta$
could possibly be created, i.e., the first scenario has probability $P_1$
in (\ref{P1}) and the second scenario has probability $P_2$ in (\ref{P2}).
Note that there are
\[
  4\binom{d_i}{2}\binom{d_i-2}{2}\binom{d_j}{2}\binom{d_j-2}{2} = \frac{(d_i)_4(d_j)_4}{4}
\]
ways to choose $\alpha$ and $\beta$ incident to the vertices $i$ and $j$. In
this case, $\beta$ is created with probability $1/3$ (as there are three
ways of rewiring the loose half-edges).

There is one other case where $\beta$ can be created by rewiring, namely
when $\alpha$ and $\beta$ share a common edge (e.g.\ $s'_1=s_1$ and
$t'_1=t_1$, while the remaining half-edges are all distinct.  In that case,
rewiring will create $\beta$ precisely when $(s'_2,t'_2)$ are joined prior
to rewiring, which has probability $1/(\ell_n-1)$.  The number of
pairs $\alpha,\beta$ in this class is
\[
  (d_i)_3 (d_j)_3 .
\]

The total contribution from case (d3) is therefore 
\begin{equation}
  \label{case(dcase2)}
  \sum_{i<j \in[n]} 
  \frac{(d_i)_4(d_j)_4}{4} p_\alpha (P_1+P_2) +  \sum_{i<j \in[n]} (d_i)_3 (d_j)_3 p_\alpha
  \cdot\frac{1}{\ell_n-1}
  = O((\mu_n^{\sss(4)})^2/\ell_n^3) +  O((\mu_n^{\sss(3)})^2/\ell_n).
\end{equation}

\paragraph{Case (d4): The multiple edges $\alpha$ and $\beta$ are
  incompatible and incident to only two vertices:}
In all other configurations of $\alpha$ and $\beta$ involving only two
vertices $i,j$, $\alpha$ and $\beta$ cannot coexist, and so rewiring
destroys $\beta$ whenever it is present, which occurs with probability
$p_\beta = \frac{1}{(\ell_n-1)(\ell_n-3)}$.

This can occur in various ways: Using four half-edges from $i$ and two or
three from $j$ (or the other way around), or having some overlap in both
$i$ and $j$.  Instead of carefully enumerating all the ways this can
happen, let us just observe that the number is dominated by
$O(d_i^4 d_j^3 + d_i^3 d_j^4)$, since at most three half-edges are chosen
at one vertex and at most four at the other.  Since $i,j$ may be swapped,
the total contribution from this case is at most
\begin{equation}
  \label{case(dcase2number)}
  \sum_{i\neq j \in[n]} C d_i^3 d_j^4 p_\alpha p_\beta
  = O(\mu_n^{\sss(4)} \mu_n^{\sss(3)}/\ell_n^2).
\end{equation}

\bigskip

In total, the contribution due to case (d) is thus equal to 
\eqn{
	\label{case(dpart1)}
	O(\nu_n^4/\ell_n)+O(\mu_n^{\sss(4)}\nu_n^2/\ell_n^2)+O(\mu_n^{\sss(3)}\nu_n^2/\ell_n)+O(\nu_n^3/\ell_n)+O((\mu_n^{\sss(3)})^2/\ell_n)
	= O(1) \frac{\nu_n^4+(\mu_n^{\sss(3)})^2}{\ell_n}.
      }
Here we use that $\nu_n^2\mu_n^{\sss(3)}\leq \nu_n^4+(\mu_n^{\sss(3)})^2$
and $\nu_n=O(\ell_n)$.  Thus, we note that the largest contributions in
case (d) are due to one sum that appears in case (d1) and one sum that
appears in case (d3), i.e.,
\eqn{
	\label{case(d)}
	O(1)\sum_{i\neq j\neq k\neq l \in[n]} \frac{(d_i)_2(d_j)_2(d_k)_2(d_l)_2}{(\ell_n-1)^2(\ell_n-3)^2(\ell_n-5)}+O(1)\sum_{i\neq j \in[n]}\frac{(d_i)_3(d_j)_3}{(\ell_n-1)^2(\ell_n-3)}
	=O(1) \frac{\nu_n^4+(\mu_n^{\sss(3)})^2}{\ell_n}.
	}

\section{Proofs of main theorems}
\label{sec-proofs}

\subsection{Proofs for configuration model}
	
\paragraph{Conclusion to the proof of Theorem \ref{thm-PA-general}.}
To conclude the proof of Theorem \ref{thm-PA-general}, we distinguish
between the proofs of \eqref{thm-1}, \eqref{thm-2} and \eqref{thm-3}, and
note that \eqref{thm-4} is a direct consequence of \eqref{thm-3}. For each
of these cases, we need to sum up the corresponding contributions in the
above cases (a)-(d).  \medskip

To prove \eqref{thm-1}, we only need to consider the contribution due to
case (a), which is $O(\nu_n^2/\ell_n)$. The contribution due to
$\sum_\alpha p_\alpha^2$ equals $O(\nu_n/\ell_n)=O(\nu_n^2/\ell_n)$, while
$\lambdanS=(\nu_n/2)(1+O(1/n))$. Thus, Theorem \ref{thm-PA-sums} gives that
	\eqn{
	\dtv{\Law(S_n)-\Po(\lambdanS)}\leq \frac{O(1)}{(\nu_n/2\vee 1)}\frac{\nu_n^2}{\ell_n},
	}
which completes the proof of \eqref{thm-1}.

To prove \eqref{thm-2}, we only need to consider the contribution due to
case (d), which is $O((\nu_n^4+(\mu_n^{\sss(3)})^2)/\ell_n)$. The
contribution due to $\sum_\alpha p_\alpha^2=\lambdanM/(\ell_n-1)(\ell_n-3)$
equals $O(\nu_n^2/\ell_n^2)$. Thus, Theorem \ref{thm-PA-sums} gives that
	\eqn{
	\dtv{\Law(M_n)-\Po(\lambdanM)}\leq \frac{O(1)}{(\lambdanM\vee 1)}\frac{\nu_n^4+(\mu_n^{\sss(3)})^2}{\ell_n},
	}
which completes the proof of \eqref{thm-2}.

To prove \eqref{thm-3}, we need to consider the contribution due to cases
(a)-(d), which is $O((\nu_n^4+(\mu_n^{\sss(3)})^2)/\ell_n)$. The
contribution due to $\sum_\alpha
p_\alpha^2=\lambdanS/(\ell_n-1)+\lambdanM/(\ell_n-1)(\ell_n-3)$ equals
$O(\nu_n/\ell_n)$. Thus, Theorem \ref{thm-PA-sums} gives that
	\eqn{
	\dtv{\Law(S_n+M_n)-\Po(\lambdanS+\lambdanM)}\leq \frac{O(1)}{((\lambdanS+\lambdanM)\vee 1)}\frac{\nu_n^4+(\mu_n^{\sss(3)})^2}{\ell_n},
	}
which completes the proof of \eqref{thm-3}, and thus of Theorem \ref{thm-PA-general}.
\qed

\paragraph{Conclusion to the proof of Theorem \ref{thm-PA-fin-var}.}
For Theorem \ref{thm-PA-fin-var}, we use the Poisson approximation for $W$ in \eqref{CW-variable}, and rely on Corollary \ref{cor-PCW}. 
Since $\lim_{n\rightarrow \infty}\expec[D_n^2]=\expec[D^2]<\infty,$ $\dmax=o(\sqrt{n})$, so that $(\mu_n^{\sss(3)})^2/\ell_n\leq \dmax^2\nu_n^2/\ell_n=o(1)$. Thus, $W\convd W_p$, which has a Poisson distribution with parameter $p\nu/2+q\nu^2/4$, so that the assumptions in Corollary \ref{cor-PCW} are satisfied. We conclude that $(S_n,M_n)\convd (S,M)$, where $S$ and $M$ are independent Poisson variables with parameters $\nu/2$ and $\nu^2/4$ respectively. This implies that $\dtv{\Law(S_n,M_n)-\Po(\nu/2)\otimes \Po(\nu^2/4)}\to 0$, since for integer-valued random vectors, the two notions of convergence are equivalent.
\qed

\paragraph{Conclusion to the proof of Theorems \ref{thm-PA-fin-var-b}--\ref{thm-PA-fin-var-c}.}
For Theorem \ref{thm-PA-fin-var-b}, we note that $\nu_n\rightarrow \nu$ and
$\mu_n^{\sss(3)}\rightarrow \mu^{\sss(3)}\equiv \expec[(D)_3]/\expec[D]$
under the assumptions of Theorem \ref{thm-PA-fin-var-b}. For Theorem
\ref{thm-PA-fin-var-c}, we note that $\nu_n\rightarrow \nu$ under the
assumptions of Theorem \ref{thm-PA-fin-var-c}, while $\mu_n^{\sss(3)}\leq
\dmax \nu_n$.
\qed

\paragraph{Conclusion to the proof of Theorems \ref{thm-PA-inf-var}--\ref{thm-PA-inf-var-b}.}
Theorem \ref{thm-PA-inf-var} follows from the fact that $\nu_n\rightarrow
\infty$, so that the bound in \eqref{thm-1} is $O(\nu_n/\ell_n)$. Since
$\nu_n\leq \dmax$ and $\dmax=o(n)$ when $\expec[D_n]\rightarrow \expec[D]$,
we obtain that $\dtv{\Law(S_n)-\Po(\lambdanS)}=o(1)$. Since
$\nu_n\rightarrow \infty$, by the CLT,
$(\Po(\lambdanS)-\lambdanS)/\sqrt{\lambdanS}\convd Z$. Since
$\lambdanS=(\nu_n/2)(1+O(1/n))$, this completes the proof.

The proof of Theorem \ref{thm-PA-inf-var-b} is similar, now using $\lambdanM=\Theta(\nu_n^2)\rightarrow \infty$, so that
	\eqn{
	\frac{O(1)}{((\lambdanS+\lambdanM)\vee 1)}\frac{\nu_n^4+(\mu_n^{\sss(3)})^2}{\ell_n}
	\leq \frac{\nu_n^2+(\mu_n^{\sss(3)}/\nu_n)^2}{\ell_n}\leq \dmax^2/\ell_n=o(1),
	}
since we assume that $\dmax=o(\sqrt{n})$.
\qed

\subsection{Proofs for directed and bipartite configuration models}
\label{sec-dir-bi-CM-pfs}
\paragraph{Conclusion to the proof of Theorem \ref{thm-PA-dir}.}	

The proof is very similar to the proof of Theorem \ref{thm-PA-general}.
We again distinguish between
the proofs of \eqref{thm-1d}, \eqref{thm-2d} and \eqref{thm-3d}, and note
that \eqref{thm-4d} is a direct consequence of \eqref{thm-3d}. For each of
these cases, we again need to sum up the corresponding contributions (of
the couplings) in the above cases (a)-(d), but now for the directed
configuration model $\DCMnd$ (instead of $\CMnd$). 
Below, we abbreviate $ \hat{\nu}_n= \hat{\lambda}_n^{\hS}$, $\hat{\xi}_n= \hat{\lambda}_n^{\hM}$.

To prove \eqref{thm-1d}, we only need to consider the contribution due to
case (a), which now is $O(\hat{\nu}_n^2/\hat{\ell}_n)$. Again the main
contribution is when the self-loops $\alpha$ and $\beta$ are incident to
two distinct vertices $i$ and $j$, and $ \beta $ is created.
Note that $ p_{\alpha}=1/\hat{\ell}_n$. 
The first sum in \eqref{case(a)} (which was the main contribution in the
undirected model) now corresponds to
\[
  \sum_{i\neq j\in[n]}
  \frac{d_i^{\sss \rm (in)}d_i^{\sss \rm (out)}d_j^{\sss \rm (in)}d_j^{\sss
      \rm (out)}}{\hat{\ell}^{2}_n(\hat{\ell}_n-1)}=O(\hat{\nu}_n^2/\hat{\ell}_n).
\]

The contribution due to $\sum_\alpha p_\alpha^2$ equals
$O(\hat{\nu}_n/\hat{\ell}_n)=O(\hat{\nu}_n^2/\hat{\ell}_n)$, while
$\hat{\lambda}_n^{\hS}=\hat{\nu}_n(1+O(1/n))$. Thus, Theorem
\ref{thm-PA-sums} gives that
	\eqn{
	\dtv{\Law(\hat{S}_n)-\Po(\hat{\lambda}_n^{\hS})}\leq \frac{C}{(\hat{\lambda}_n^{\hS}\vee 1)}\frac{\hat{\nu}_n^2}{\hat{\ell}_n},
	}
which completes the proof of \eqref{thm-1d}.

To prove \eqref{thm-2d}, we only need to consider the contribution due to case (d), which now is $O(\frac{\mu_n^{\sss(3,\rm{in})}\mu_n^{\sss(3,\rm{out})}+\hat{\xi}_n^2}{\hat{\ell}_n})$. 
Note that $ p_{\alpha}=\frac{1}{\hat{\ell}_n(\hat{\ell}_n-1)} $.
The contribution due to $\sum_\alpha p_\alpha^2=\hat{\lambda}_n^{\hM}/\hat{\ell}_n(\hat{\ell}_n-1)$ equals $O(\hat{\xi}_n/\hat{\ell}_n^2)$. 
There are two main contributions corresponding to the two main contributions in the undirected model. 
The first main contribution corresponds to the case when $\alpha$ and   $\beta$ are incident to four different vertices, and $ \beta $ is created. Then the corresponding sum in (\ref{case(dcase1)}) is now
$$ \sum_{i\neq j\neq k\neq l \in[n]} O(\frac{(d_i^{\sss \rm (in)})_2(d_j^{\sss \rm (out)})_2(d_k^{\sss \rm (in)})_2(d_l^{\sss \rm (out)})_2}{\hat{\ell}_n^2(\hat{\ell}_n-1)^2(\hat{\ell}_n-2)}= O(\frac{\hat{\xi}_n^2}{\hat{\ell}_n}).$$
The second main contribution corresponds to the case when $\alpha$ and $ \beta $ are incident to two vertices $ i $ and $ j $, and $ \beta $ is created. Then the corresponding second sum in (\ref{case(dcase2)}) is now $ \sum_{i\neq j \in[n]} O(\frac{(d_i^{\sss \rm (in)})_3(d_j^{\sss \rm (out)})_3}{\hat{\ell}_n^2(\hat{\ell}_n-1)})=O(\frac{\mu_n^{\sss(3,\rm{in})}\mu_n^{\sss(3,\rm{out})}}{\hat{\ell}_n}).$

Thus, Theorem \ref{thm-PA-sums} gives that
	\eqn{
	\dtv{\Law(\hat{M}_n)-\Po(\hat{\lambda}_n^{\hM})}\leq \frac{C}
	{(\hat{\lambda}_n^{\hM} \vee 1)}\frac{\mu_n^{\sss(3,\rm{in})}\mu_n^{\sss(3,\rm{out})}+\hat{\xi}_n^2}{\hat{\ell}_n},
	}
which completes the proof of \eqref{thm-2d}.

To prove \eqref{thm-3d}, we need to consider the contribution due to cases (a)-(d), which now is $O(\frac{\mu_n^{\sss(3,\rm{in})}\mu_n^{\sss(3,\rm{out})}+\hat{\xi}_n^2}{\hat{\ell}_n})$. The contribution due to $\sum_\alpha p_\alpha^2=\hat{\lambda}_n^{\hS}/\hat{\ell}_n+\hat{\lambda}_n^{\hM}/\hat{\ell}_n(\hat{\ell}_n-1)$ equals $O(\hat{\nu}_n/\hat{\ell}_n)$+$O(\hat{\xi}_n/\hat{\ell}_n^2)$. Thus, Theorem \ref{thm-PA-sums} gives that
	\eqn{
	\dtv{\Law(\hat{S}_n+\hat{M}_n)-\Po(\hat{\lambda}_n^{\hS}+\hat{\lambda}_n^{\hM})}\leq \frac{C}{((\hat{\lambda}_n^{\hS}+\hat{\lambda}_n^{\hM})\vee 1)}\frac{\mu_n^{\sss(3,\rm{in})}\mu_n^{\sss(3,\rm{out})}+\hat{\xi}_n^2}{\hat{\ell}_n},
	}
which completes the proof of \eqref{thm-3d}, and thus of Theorem \ref{thm-PA-dir}.
\qed

\paragraph{Conclusion to the proof of Theorem \ref{thm-PA-bip}.}

The proof is again very similar to the proof of Theorem \ref{thm-PA-general} and
that of Theorem \ref{thm-PA-dir}. However, there are no self-loops in the
bipartite configuration model. Thus, we only need to consider case (d)
above (regarding the couplings for the multiple edges), but now for the
bipartite configuration model $\BCMnd$. 
Again there are two main contributions corresponding to the main
contributions for the undirected configuration model $\CMnd$. 

Note that $ p_{\alpha}=\frac{1}{\bar{\ell}_n(\bar{\ell}_n-1)} $.
The corresponding sum in (\ref{case(dcase1)}) is now
	$$ 
	\sum_{i,k\in [n^{\sss \rm(l)}], j,l \in[n^{\sss \rm(r)}]} O(\frac{(d_i^{\sss\rm (l)})_2(d_j^{\sss\rm (r)})_2(d_k^{\sss\rm (l)})_2(d_l^{\sss\rm (r)})_2}
	{\bar{\ell}_n^2(\bar{\ell}_n-1)^2(\bar{\ell}_n-2)}= O\Big(\frac{(\bar{\lambda}_n^{\bM})^2}{\bar{\ell}_n}\Big).
	$$
The corresponding second sum in (\ref{case(dcase2)}) is now
\[
  \sum_{i\in
  [n^{\sss \rm(l)}], j\in[n^{\sss \rm(r)}]} O(\frac{(d_i^{\sss\rm
    (l)})_3(d_j^{\sss\rm
    (r)})_3}{\bar{\ell}_n^2(\bar{\ell}_n-1)})=O(\frac{\mu_n^{\sss(3,\rm{l})}\mu_n^{\sss(3,\rm{r})}}{\bar{\ell}_n}).
\]

Thus, Theorem \ref{thm-PA-sums} gives that
	\eqn{
	\dtv{\Law(\bar{M}_n)-\Po(\bar{\lambda}_n^{\bM})}\leq \frac{C}
	{(\bar{\lambda}_n^{\bM} \vee 1)}\frac{\mu_n^{\sss(3,\rm{l})}\mu_n^{\sss(3,\rm{r})}+(\bar{\lambda}_n^{\bM})^2}{\bar{\ell}_n},
	}
which completes the proof of \eqref{thm-2b}, and thus of Theorem \ref{thm-PA-bip}.
\qed
\medskip

\paragraph{Acknowledgements.}
The work of OA is supported by NSERC, as well as by the Isaac Newton
Institute and the Simons Foundation.
The work of RvdH is supported by the Netherlands Organisation for
Scientific Research (NWO) through VICI grant 639.033.806 and the
Gravitation {\sc Networks} grant 024.002.003.  The work of CH is supported
by a grant from the Swedish Research Council.  This research was initiated
during the workshop ``Probability, Combinatorics and Geometry'' at the
Bellairs Research Institute in Barbados, April 2014.  
We would like to give special thanks to Konstantinos Panagiotou for many stimulating discussions during this workshop. We would also like to thank the Bellairs Research Institute as well as the organizers of this workshop, 
Louigi Addario-Berry, Nicolas Broutin, Luc Devroye, Vida Dujmovic and G\'abor
Lugosi.

\bibliographystyle{plainnat}

\def\cprime{$'$}


\end{document}